\documentclass[12pt]{amsart}
\usepackage{fancyhdr}
\usepackage{stmaryrd}
\usepackage{calrsfs}

\usepackage{amsmath}
\usepackage[nospace,noadjust]{cite}
\usepackage{amsfonts}
\usepackage{amssymb,enumerate}
\usepackage{amsthm}
\usepackage{cite}
\usepackage{comment}
\usepackage{color}
\usepackage[all]{xy}
\usepackage{hyperref}
\usepackage{lineno}
\usepackage{tikz-cd}

\newtheorem*{maintheorem*}{Main Theorem}
\newtheorem{theorem}{Theorem}[section]
\newtheorem{prop}[theorem]{Proposition}

\newtheorem{lemma}[theorem]{Lemma}
\newtheorem{cor}[theorem]{Corollary}

\theoremstyle{definition}

\newtheorem{remark}[theorem]{Remark}
\newtheorem{example}[theorem]{Example}
\numberwithin{equation}{section}

\newcommand{\nn}{\mathbb{N}}
\newcommand{\pp}{\mathbb{P}}
\newcommand{\qq}{\mathbb{Q}}
\newcommand{\rr}{\mathbb{R}}
\newcommand{\zz}{\mathbb{Z}}

\newcommand{\uu}{\mathcal{U}}
\providecommand\ldb{\llbracket}
\providecommand\rdb{\rrbracket}

\hypersetup{
	pdftitle={OHF},
	colorlinks=true,
	linkcolor=blue,
	citecolor=cyan,
	urlcolor=wine
}



\begin{document}
	
	\mbox{}
	\title{The Furstenberg property in Puiseux monoids}
	
	\author{Andrew Lin}
	\address{Hopewell Valley Central High School\\Pennington, NJ 08534}
	\email{andrewlin@hvrsd.org}

	\author{Henrick Rabinovitz}
	\address{Melrose High School\\Melrose, MA 02176}
	\email{hrey77@gmail.com}

	\author{Qiao Zhang}
	\address{Sierra Canyon School\\Chatsworth, CA 91311}
	\email{tigerzhang343@gmail.com}

\date{\today}

\begin{abstract}
	 Let $M$ be a commutative monoid. The monoid $M$ is called atomic if every non-invertible element of $M$ factors into atoms (i.e., irreducible elements), while $M$ is called a Furstenberg monoid if every non-invertible element of $M$ is divisible by an atom. Additive submonoids of $\qq$ consisting of nonnegative rationals are called Puiseux monoids, and their atomic structure has been actively studied during the past few years. The primary purpose of this paper is to investigate the property of being Furstenberg in the context of Puiseux monoids. In this direction, we consider some properties weaker than being Furstenberg, and then we connect these properties with some atomic results which have been already established for Puiseux monoids.
\end{abstract}
\medskip

\maketitle

\bigskip
%%%%%%%%%%%
%%%%%%%%%%%
\section{Introduction}
\label{sec:intro}

An element in a commutative monoid is atomic if it is invertible or it factors into atoms (i.e., irreducible elements). Following Cohn~\cite{pC68}, we say that a commutative monoid is atomic if every element is atomic. Before 1990, atomicity was studied in its own only sporadically and mainly in connection to ascending chains of principal ideals (see, for instance, \cite{aG74,aZ82}). A more systematic study of atomicity started right after the paper \cite{AAZ90}, where the authors introduced some notions stronger than atomicity and provided the first dedicated investigation of factorizations in the context of integral domains. Divisibility properties weaker than atomicity have also been considered in recent years: see, for instance, \cite{CZ04,BC15} and the more recent paper~\cite{nLL19}. The Furstenberg property, which is one of the divisibility properties we consider here, was coined by Clark~\cite{pC17}. A commutative monoid is said to be Furstenberg if every non-invertible element is divisible by an atom.  It follows directly from the definitions that every atomic monoid is Furstenberg. The Furstenberg property has been recently considered in~\cite{nLL19,GZ22} in the context of integral domains, in~\cite{GL22} for rings of integer-valued polynomials, and in~\cite{GP23} in the context of polynomial semidomains.
\smallskip

The primary purpose of this paper is to provide a better understanding of the atomic structure of non-atomic Puiseux monoids, and we do so by studying some weaker notions of the Furstenberg property, which were already considered in~\cite{nLL19} for integral domains. A Puiseux monoid is a subset of $\qq$ that consists of nonnegative rationals, contains~$0$, and is closed under addition. Although Puiseux monoids played some isolated roles in the 1990s in the setting of factorizations in integral domains (see, for instance, \cite[Example~2.1]{AAZ90} and~\cite[Example~7]{AM96}), they were not systematically considered in the literature after the appearance of the paper~\cite{fG17}, where Puiseux monoids are first investigated in their own right. Puiseux monoids are perhaps the most elementary algebraic structures showing an intricate atomic structure. Because of this, Puiseux monoids have been actively considered in recent years as more interest in the atomic theory of algebraic structures has emerged (see \cite{GGT21} and references therein). It is worth emphasizing that although the atomic structure of atomic Puiseux monoids has been actively investigated during the past few years, the same cannot be said for non-atomic Puiseux monoids, even though some useful and interesting Puiseux monoids are not atomic (see, for instance, \cite[Example~4.11]{GZ22} and \cite[Section~4]{GV23}).
\smallskip

Besides being studied in their own right, Puiseux monoids have found applications and been investigated beyond semigroup theory. Perhaps the first relevant application of Puiseux monoids goes back to the 1970s, where the Puiseux monoid generated by the sequence $\big( \frac{1}{2^n p_n}\big)_{n \ge 1}$ was used by Grams~\cite{aG74} as the main ingredient to construct the first atomic integral domain that does not satisfy the ascending chain condition on principal ideals. Most of the applications of Puiseux monoids can be found in semigroup algebras: for instance, they have been used by Coykendall and Gotti~\cite[Theorem~5.4]{CG19} to show that atomicity does not ascend from monoids to monoid algebras, and they have been used by Gotti and Zafrullah~\cite[Example~4.11]{GZ22} to argue that the irreducible-divisor-finite property (introduced by Grams and Warner in~\cite{GW75}) does not ascend from monoids to monoid algebras. In addition, the paper~\cite{fG22} is devoted to the study of factorizations in monoid algebras constructed from Puiseux monoids. Recent constructions in commutative ring theory involving Puiseux monoids can be found in~\cite{GL23}. Finally, Puiseux monoids were investigated by Baeth and Gotti~\cite{BG20} in the setting of factorization of matrices, while Puiseux monoids were considered by Bras-Amor\'os~\cite{mB20} in connection with music theory.
\smallskip

Let $M$ be a commutative monoid. Following Coykendall, Dobbs, and Mullins~\cite{CDM99}, we say that $M$ is antimatter if the set of atoms of $M$ is empty. The monoid $M$ is called nearly Furstenberg if there exists $c \in M$ such that for all non-invertible $b \in M$ there exists an atom $a \in M$ such that $a$ divides $bc$ in $M$ but $a$ does not divide $c$ in $M$. Also, $M$ is called quasi-Furstenberg (resp., almost Furstenberg) if for each non-invertible $b \in M$, there exists an element $c \in M$ (resp., an atomic element $c \in M$) and an atom $a \in M$ such that $a$ divides $bc$ in $M$ but $a$ does not divide $c$ in $M$. It follows directly from the given definitions that every Furstenberg monoid is both nearly Furstenberg and almost Furstenberg and also that every nearly or almost Furstenberg monoid is quasi-Furstenberg. The implications in Diagram~\eqref{eq:Furstenberg diagram} illustrate this observations (for simplicity, we have use the non-standard acronyms FM, NFM, AFM, and QFM in the same diagram for Furstenberg, nearly Furstenberg, almost Furstenberg, and quasi-Furstenberg monoids, respectively). The weaker notions of the Furstenberg property we have introduced in this paragraph were all introduced and first studied by Lebowitz-Lockard in~\cite{nLL19}.

\begin{equation} \label{eq:Furstenberg diagram}
	\begin{tikzcd}%[cramped]
		\textbf{ FM } \ \arrow[r, Rightarrow] \arrow[red, r, Leftarrow, "/"{anchor=center,sloped}, shift left=1.7ex] \arrow[d, Rightarrow, shift right=1ex] \arrow[red, d, Leftarrow, "/"{anchor=center,sloped}, shift left=1ex]& \ \textbf{ NFM } \arrow[d, Rightarrow, shift right=0.6ex] \arrow[red, d, Leftarrow, "/"{anchor=center,sloped}, shift left=1.3ex] \\
		\textbf{ AFM } \ \arrow[r, Rightarrow]	\arrow[red, r, Leftarrow, "/"{anchor=center,sloped}, shift left=1.7ex] & \ \textbf{ QFM } \arrow[r, Rightarrow]  \arrow[red, r, Leftarrow, shift left=1.7ex] & \textbf{ non-antimatter}
	\end{tikzcd}
\end{equation}

This paper is organized as follows. In Section~\ref{sec:background}, we introduce and briefly discuss most of the notation, terminology, and known non-standard results relevant to the development of the subsequent sections of content. In Section~\ref{sec:Furstenbergness}, we study in the setting of Puiseux monoids the properties shown in Diagram~\eqref{eq:Furstenberg diagram}. Among other findings, we show that none of the implications represented by broken (red) arrows in Diagram~\eqref{eq:Furstenberg diagram} hold in the class of Puiseux monoids. We also prove a result that yields as a corollary that every Puiseux monoid that is not antimatter must be quasi-Furstenberg, which is represented in Diagram~\eqref{eq:Furstenberg diagram} by the rightmost horizontal reversed (red) arrow. Note that the parallel rightmost horizontal black arrow assumes that the monoid is not a group. In Section~\ref{sec:atomicity}, we provide further examples to help our audience understand the inter-connection of the properties already studied in Section~\ref{sec:Furstenbergness} with the atomic structure of non-atomic Puiseux monoids. We construct a Puiseux monoid that is Furstenberg but not almost atomic as well as a Puiseux monoid that is nearly atomic but not Furstenberg (see Section~\ref{sec:atomicity} for the definitions of nearly and almost atomicity).

\bigskip
%%%%%%%%%%%%
%%%%%%%%%%%%
\section{Background}
\label{sec:background}

In this section, we introduce some terminology and definitions mostly related to the atomicity of cancellative commutative monoids.

\smallskip
%%%%%%%%%%%%%%%%
\subsection{General Notation}

We use $\nn$ and $\nn_0$ to denote the sets of positive and non-negative integers, respectively. Also, we let $\pp$ denote the set of primes. Following standard notation, we let $\zz$, $\qq$, and $\rr$ denote the sets of integers, rationals, and real numbers, respectively. For $r \in \rr$ and $X \subseteq \rr$, we set $X_{\ge r} := \{x \in X \mid x \ge r\}$. For $a,b \in \rr$ with $a \le b$, we set
\[
	\ldb a, b \rdb := \{n \in \zz \mid a \le n \le b\}.
\]
For $q \in \qq_{> 0}$, we let $\mathsf{n}(q)$ and $\mathsf{d}(q)$ denote the unique pair of elements of $\nn$ satisfying $\gcd(\mathsf{n}(q), \mathsf{d}(q)) = 1$ and $q = \frac{\mathsf{n}(q)}{\mathsf{d}(q)}$: in this case, we call $\mathsf{n}(q)$ and $\mathsf{d}(q)$ the \emph{numerator} and \emph{denominator} of $q$, respectively. For a subset $Q$ of $\qq_{>0}$, we call the set
\[
	\{p \in \pp \mid p \text{ divides } \mathsf{d}(q) \text{ for some } q \in Q\}
\] 
the \emph{prime support} of $Q$. For $p \in \pp$ and $n \in \nn$, the value $v_p(n)$ is the exponent of the maximal power of $p$ dividing $n$. In addition, the $p$-\emph{adic valuation} is the map $v_p \colon \qq \to \zz$ defined by $v_p(q) = v_p(-q) = v_p(\mathsf{n}(q)) - v_p(\mathsf{d}(q))$ for $q \in \qq_{> 0}$ and $v_p(0) = \infty$. For any two sets $S$ and $T$, we let $S \Delta T$ denote the symmetric difference of $S$ and $T$; that is, $S \Delta T = (S \setminus T) \cup (T \setminus S)$. %It can be easily seen that the $p$-adic valuation satisfies that $v_p(q_1 + \dots + q_n) \ge \min\{v_p(q_1), \dots, v_p(q_n) \}$ for every $n \in \nn$ and $q_1, \dots, q_n \in \qq_{> 0}$.

\smallskip
%%%%%%%%%%%%%%%%%%%
\subsection{Commutative Monoids}

Over the course of this paper, the word \textit{monoid} will refer to a cancellative commutative monoid. Let $M$ be a monoid. Since $M$ is assumed to be commutative, we will always use additive notation: ``$+$" denotes the operation of $M$ and $0$ denotes the identity element (note the contrast with the multiplicative notation we have used in the introduction). If $M = \{0\}$, then we say that $M$ is \emph{trivial}. We denote the set $M \setminus \{0\}$ as $M^{\bullet}$. The group consisting of all invertible elements of $M$ is denoted by $\uu(M)$. The monoid $M$ is called \emph{reduced} if $\uu(M) = \{0\}$. A \emph{submonoid} of $M$ is a subset of $M$ that contains $0$ and is closed under addition. We say that $M$ is \emph{generated} by a subset $S$ provided that the smallest submonoid (under inclusion) of $M$ containing $S$ is $M$ itself: in this case, we write $M = \langle S \rangle$. If there exists a finite subset $S$ of $M$ such that $M = \langle S \rangle$, then $M$ is called \emph{finitely generated}. For general background in commutative monoids, see~\cite{pG01}.

An element $a \in M \setminus \uu(M)$ is called an \emph{atom} (or an \emph{irreducible}) if when the equality $a = x + y$ holds for some $x,y \in M$, then either $x \in \uu(M)$ or $y \in \uu(M)$. We let $\mathcal{A}(M)$ denote the set of atoms of $M$. When $\mathcal{A}(M)$ is empty, we say that $M$ is \emph{antimatter}. An element $b \in M$ is called \emph{atomic} if either $b \in \uu(M)$ or $b$ can be written as a sum of atoms. The monoid $M$ is called \emph{atomic} if every element of $M$ is atomic. For $b,c \in M$, we say that $c$ \emph{divides} $b$ in $M$ if there exists $d \in M$ such that $b = c+d$, in which case, we write $c \mid_M b$. An element $b \in M$ is called a \emph{Furstenberg element} if either $b \in \uu(M)$ or $b$ is divisible by an atom in $M$. The monoid $M$ is called a \emph{Furstenberg monoid} provided that every element of $M$ is a Furstenberg element. It follows directly from the definitions that every atomic element is a Furstenberg element and, therefore, every atomic monoid is a Furstenberg monoid.

A \emph{Puiseux monoid} is an additive submonoid of $\mathbb{Q}_{\ge 0}$. The atomic structure of Puiseux monoids have been systematically investigated during the past few years (see \cite{CGG20,GGT21} and references therein). A special subclass of Puiseux monoids is that consisting of numerical monoids: a \emph{numerical monoid} is an additive submonoid $N$ consisting of nonnegative integers such that $\nn_0 \setminus N$ is finite. Although every numerical monoid is finitely generated, and so atomic in light of \cite[Proposition~2.7.8]{GH06}, the same does not hold for Puiseux monoids in general. For instance, $\qq_{\ge 0}$ is a non-atomic Puiseux monoid with no atoms, while $\{0\} \cup \qq_{\ge 1}$ is an atomic Puiseux monoid whose set of atoms is the infinite set $\qq \cap [1,2)$. In addition, not every Puiseux monoid is Furstenberg (as it is the case of $\qq_{\ge 0}$) and there are Puiseux monoids that are Furstenberg but not atomic, as we shall see in Section~\ref{sec:atomicity}. See \cite{CGG20a} for a survey on the atomicity of Puiseux monoids.

One can check that if $\varphi \colon M \to M'$ is a homomorphism between two Puiseux monoids $M$ and $M'$, then there must exist $q \in \qq_{\ge 0}$ such that $M' = qM$; that is, the homomorphisms between Puiseux monoids are given by rational multiplication. Using this, we can argue the following lemma, which will be referenced several times throughout the paper.

\begin{lemma} \label{lem:prelim isomorphism}
	Let $M$ and $M'$ be two isomorphic Puiseux monoids, and let $P$ and $P'$ be the prime supports of $M^\bullet$ and $M'^\bullet$ respectively. Then the following statements hold.
	\begin{enumerate}
		\item $|P \Delta P'| < \infty$.
		\smallskip
		
		\item $\inf v_p(M^\bullet) = -\infty$ if and only if $\inf v_p(M'^\bullet) = -\infty$ for every $p \in \pp$. 
	\end{enumerate}
\end{lemma}

\begin{proof}
	(1) Since $M$ and $M'$ are isomorphic Puiseux monoids, \cite[Proposition~3.2]{fG18} guarantees the existence of $q \in \qq_{> 0}$ such that $M' = qM$ (if $q=0$, then both $M$ and $M'$ must be trivial monoids and then we can replace $q$ be any positive rational). Therefore $M'^\bullet=qM^\bullet$, and so $v_p(M'^\bullet) = v_p(q) + v_p(M^\bullet)$ for any prime $p$. Now, if $p\in P\Delta P'$, then one of the sets $v_p(M^\bullet)$ and $v_p(M'^\bullet)$ is a subset of $\nn_0$ and the other is not; in particular, $v_p(M^\bullet)\ne v_p(M'^\bullet)=v_p(q)+v_p(M^\bullet)$, so $v_p(q)\ne 0$. As there are only finitely many primes $p$ such that $v_p(q)\ne 0$, the symmetric difference $P\Delta P'$ must be finite.
	\smallskip
	
	(2) Fix a prime $p$. From the fact that $v_p(M'^\bullet)=v_p(q)+v_p(M^\bullet)$, we obtain that $\inf v_p(M'^\bullet)=v_p(q)+\inf v_p(M^\bullet)$. Thus, if one of $\inf v_p(M'^\bullet)$ and $\inf v_p(M^\bullet)$ is greater than $-\infty$, then so is the other.
\end{proof}

\bigskip
%%%%%%%%%%%%%%
%%%%%%%%%%%%%%
\section{Furstenbergness}
\label{sec:Furstenbergness}

\medskip
%%%%%%%%%%%%%%%%%%%%%%%%%%%
\subsection{Quasi-Furstenberg Puiseux Monoids} Let $M$ be a monoid. We say that $M$ is \emph{quasi-Furstenberg} if for each $b \in M \setminus \uu(M)$, there exists $c \in M$ and $a \in \mathcal{A}(M)$ such that $a \mid_M b+c$ but $a \nmid_M c$. We say that an element $b \in M$ is \emph{quasi-atomic} if there exists an element $c \in M$ such that $b+c$ is atomic. The monoid $M$ is called \emph{quasi-atomic} if every element of $M$ is quasi-atomic. The notion of quasi-atomicity was introduced by Boynton and Coykendall~\cite{BC15}.

It turns out that there is a large amount of quasi-Furstenberg Puiseux monoids. Indeed, as the following proposition shows, every Puiseux monoid containing at least one atom is quasi-atomic, and every quasi-atomic (Puiseux) monoid is quasi-Furstenberg.

\begin{prop} \label{prop:antimatter or QA}
	The following statements hold.
	\begin{enumerate}
		\item Every quasi-atomic monoid is quasi-Furstenberg.
		\smallskip
		
		\item Every Puiseux monoid is antimatter or quasi-atomic.
		\smallskip
		
		\item A nontrivial Puiseux monoid is quasi-Furstenberg if and only if it is quasi-atomic.
	\end{enumerate}
\end{prop}

\begin{proof}
	(1) Suppose that $M$ is a quasi-atomic monoid. If $M$ is a group, then it follows directly from the definition that $M$ is quasi-Furstenberg. Therefore we assume that $M$ is not a group. Fix a non-invertible element $b \in M$. As $M$ is quasi-atomic, we can choose $c \in M$ such that the element $b+c$ is atomic. We can assume that we have chosen a $c$ that minimizes the number of atoms (counting repetitions) in an atomic decomposition of $b+c$. Now take the minimum $k \in \nn$ such that $b+c = a_1 + \dots + a_k$ for some $a_1, \dots, a_k \in \mathcal{A}(M)$. The minimality of both $c$ and $k$ guarantees that $a_1 \nmid_M c$. Hence $a_1$ is an atom of $M$ such that $a_1 \mid_M b+c$ but $a_1 \nmid_M c$. As a result, we can conclude that $M$ is quasi-Furstenberg.
	\smallskip
	
	(2) Let $M$ be a Puiseux monoid. Assume that $M$ is not antimatter, and let $a$ be an atom of $M$. It is clear that $0$ is a quasi-atomic element. Thus, fix a nonzero element $b \in M$, and let us argue that $b$ is quasi-atomic. To do so, first observe that the equality $\mathsf{n}(b) \mathsf{d}(a) a = \mathsf{n}(a) \mathsf{d}(b) b$ holds. Then, after setting $c :=(\mathsf{n}(a) \mathsf{d}(b) - 1)b$, we find that $b + c = \mathsf{n}(a) \mathsf{d}(b) b = \mathsf{n}(b) \mathsf{d}(a) a$. Since $a \in \mathcal{A}(M)$, the element $b+c$ is atomic and, therefore, $b$ is quasi-atomic. Hence $M$ is a quasi-atomic Puiseux monoid.
	\smallskip
	
	(3) The direct implication follows directly from part~(2) as every quasi-Furstenberg monoid that is not a group must contain at least one atom. The reverse implication follows from part~(1).
\end{proof}

Part~(2) of Proposition~\ref{prop:antimatter or QA} does not hold for classes different from the class of Puiseux monoids. The following example illustrates this observation.

\begin{example} \label{ex:Furstenberg not QA}
	Consider the additive submonoid
	\[
		M := (\nn_0 \times \{0\}) \cup (\zz \times \nn)
	\]
	of $\zz^2$, which is actually the nonnegative cone of $\zz^2$ when the latter is considered as a totally ordered group with respect to the lexicographical order with priority in the second coordinate. The element $a := (1,0)$ is the minimum nonzero element of $M$ with respect to the order of $\zz^2$ and, therefore, $a \in \mathcal{A}(M)$. As $a$ is the minimum element of $M^\bullet$, which is the positive cone of $\zz^2$, it follows that $a \mid_M b$ for every $b \in M^\bullet$. Thus, $M$ is a Furstenberg monoid with $\mathcal{A}(M) = \{a\}$. Observe that for each $b \in M$, the second coordinate of $(0,1) + b$ is positive and so $(0,1) + b \notin \nn_0 a$. As a consequence, $M$ is not quasi-atomic. Hence $M$ is neither antimatter nor quasi-atomic.
\end{example}

We can use Proposition~\ref{prop:antimatter or QA} to construct Puiseux monoids that are quasi-Furstenberg but not atomic.

\begin{example} \label{ex:a quasi-Furstenberg PM not atomic}
	Consider the Puiseux monoid $M := \big\langle \frac12, \frac1{3^n} \mid n \in \nn \big\rangle$. Observe that $\frac{1}{2}$ is an atom of $M$ as it is the only defining generating element whose denominator is divisible by the prime $2$. Since $\mathcal{A}(M) \subseteq \big\{  \frac12, \frac1{3^n} \mid n \in \nn \big\}$ and $\frac1{3^n} = 3 \cdot \frac1{3^{n+1}}$ for every $n \in \nn_0$, it follows that $\mathcal{A}(M) = \big\{ \frac12 \big\}$. Then $M$ is a quasi-Furstenberg monoid by Proposition~\ref{prop:antimatter or QA}. However, $M$ is not atomic as, for instance, $\frac13 \notin \langle \frac12 \rangle = \nn_0 \frac12$.
\end{example}

\medskip
%%%%%%%%%%%%%%%%%%%%%%%%%%%%%%%%%%%%%%%%
\subsection{Almost Furstenberg and Nearly Furstenberg Puiseux Monoids} An element $b \in M$ is called \emph{almost Furstenberg} if either $b \in \uu(M)$ or if there exists an atomic element $c \in M$ and $a \in \mathcal{A}(M)$ such that $a \mid_M b+c$ but $a \nmid_M c$. The monoid $M$ is called \emph{almost Furstenberg} if every element of $M$ is an almost Furstenberg element. It follows directly from the definitions that every Furstenberg monoid is almost Furstenberg and also that every almost Furstenberg monoid is quasi-Furstenberg.

\begin{prop} \label{prop:almost Furstenberg conditions}
	Let $M$ be a Puiseux monoid. If $|\mathcal{A}(M)| = 1$, then $M$ is almost Furstenberg if and only if $M \cong (\nn_0,+)$.
\end{prop}

\begin{proof}
	Assume that $|\mathcal{A}(M)| = 1$, and let $a$ be the only atom of $M$. The reverse implication is straightforward as the monoid $(\nn_0,+)$ is atomic. For the direct implication, suppose that $M$ is almost Furstenberg. Then, since the only atom of $M$ is $a$, for any nonzero $b\in M$, there exists an atomic element $c\in M$ such that $a\mid_M b+c$ but $a\nmid_M c$. However, for any atomic element $c \in M$ such that $a \nmid_M c$ we see that $c$ is atomic yet divisible by no atoms (since $a$ is the only atom of $M$), hence $c=0$. This means that for every nonzero $b\in M$, we must have $a\mid_M b+0=b$. Since $a$ is an atom of $M$, it is positive. Now fix a nonzero $b \in M$. Let $na$ be the largest multiple of $a$ that divides $b$ in $M$. Observe that $n$ must be a positive integer. %Notice that maximal $n$ must exist as it is bounded from above by $\left\lfloor\frac{b}{a}\right\rfloor$. 
Now, $b-na$ is an element of $M$ that is not divisible by $a$, and we have proved that the only such element is $0$, so $b=na$. Hence $M\subseteq \nn_0a$. The reverse inclusion follows from the fact that $a \in M$. Thus, $M = \nn_0 a$, which is isomorphic to $(\nn_0,+)$.
\end{proof}

Motivated by Example~\ref{ex:a quasi-Furstenberg PM not atomic}, we can use Proposition~\ref{prop:almost Furstenberg conditions} to construct quasi-Furstenberg Puiseux monoids that are not almost Furstenberg.

\begin{cor}  \label{cor:QF not AF}
	There are infinitely many non-isomorphic Puiseux monoids that are quasi-Furstenberg but not almost Furstenberg.
\end{cor}

\begin{proof}
	For each odd prime $p$, consider the Puiseux monoid $M_p := \big\langle \frac12, \frac1{p^n} \mid n \in \nn \big\rangle$. Fix $p \in \pp \setminus \{2\}$. We can proceed as in Example~\ref{ex:a quasi-Furstenberg PM not atomic} to verify that $M_p$ is a quasi-Furstenberg non-atomic monoid with $\mathcal{A}(M_p) = \big\{ \frac12 \big\}$. Since $|\mathcal{A}(M_p)| = 1$ and $M_p$ is not finitely generated, it follows from Proposition~\ref{prop:almost Furstenberg conditions} that $M_p$ is not almost Furstenberg. Finally, observe that if $q$ is an odd prime with $q \neq p$, then $\inf v_p(M_p^\bullet) = -\infty$ while $\inf v_p(M_q^\bullet) = 0$, and so it follows from part~(2) of Lemma~\ref{lem:prelim isomorphism} that $M_p$ and $M_q$ are not isomorphic Puiseux monoids.
\end{proof}

As the following proposition indicates, the class of almost Furstenberg non-Furstenberg Puiseux monoid is nonempty.

\begin{prop} \label{prop:AF not F}
	There are infinitely many non-isomorphic  Puiseux monoids that are almost Furstenberg but not Furstenberg.
\end{prop}

\begin{proof}
	For each $\ell \in \nn$, let $P_\ell$ be an infinite set consisting of odd primes. Furthermore, assume that $P_k \cap P_\ell$ is empty for all $k,\ell \in \nn$ with $k \neq \ell$. For each $\ell \in \nn$, we will construct a Puiseux monoid $M_\ell$ that is almost Furstenberg but not Furstenberg is such a way that the prime support of $M_\ell$ is an infinite subset of $\{2\} \cup P_\ell$. Then for any distinct $k,\ell \in \nn$ the symmetric difference of the prime supports of $M_k$ and $M_\ell$ will be an infinite set and, therefore, the Puiseux monoids $M_k$ and $M_\ell$ will not be isomorphic in light of part~(1) of Lemma~\ref{lem:prelim isomorphism}.
	
	To do so, fix $\ell \in \nn$. The set $\nn_0\big[\frac{1}{2}\big]_{> 1}$, consisting of all dyadic rationals greater than $1$, is countable. Hence we can pick a sequence $(b_n)_{n \ge 1}$ of positive integers and a sequence $(k_n)_{n \ge 1}$ of nonnegative integers such that $\nn_0\big[\frac{1}{2}\big]_{> 1} = \big\{\frac{b_n}{2^{k_n}} \mid n \in \nn \big\}$. For each $n \in \nn$, set $r_n := \frac{b_n}{2^{k_n}}$. Let $(p_n)_{n \in \nn}$ be a sequence of primes in $P_\ell$ such that $p_n > b_n$ for every $n \in \nn$. Now, consider the Puiseux monoid
	\[
		M_\ell := \bigg\langle \nn_0\Big[\frac12 \Big] \bigcup \bigg\{ \frac{r_n}{p_n} \ \Big{|} \ n \in \nn \bigg\} \bigg\rangle.
	\]
	We proceed to show that $M_\ell$ is almost Furstenberg but not Furstenberg. First, we will argue the following claim.
	\smallskip
	
	\noindent \textit{Claim:} $\mathcal{A}(M_\ell) = \big\{ \frac{r_n}{p_n} \mid n \in \nn \big\}$.
	\smallskip
	
	\noindent \textit{Proof of Claim:} Fix $a \in \mathcal{A}(M_\ell)$. Since $M_\ell$ is a reduced monoid, either $a \in \nn_0\big[ \frac12\big]^\bullet$ or $a \in \big\{ \frac{r_n}{p_n} \mid n \in \nn \big\}$. However, observe that if $a \in \nn_0\big[ \frac12\big]^\bullet$, then we could write $a = \frac{c}{2^k}$ for some $c \in \nn$ and $k \in \nn_0$ and so the equality $a = \frac{c}{2^{k+1}} + \frac{c}{2^{k+1}}$ would contradict that $a$ is an atom. Thus, $a \in \big\{ \frac{r_n}{p_n} \mid n \in \nn \}$ and, as a result, the inclusion $\mathcal{A}(M_\ell) \subseteq \big\{ \frac{r_n}{p_n} \mid n \in \nn \big\}$ holds. We proceed to argue the reverse inclusion. Assume, for the sake of a contradiction, that $\frac{r_j}{p_j} \notin \mathcal{A}(M_\ell)$ for some $j \in \nn$, and write
	\begin{equation} \label{eq:auxi}
		\frac{b_j}{2^{k_j} p_j} = c_0\frac{b}{2^k} + \sum_{i=1}^n c_{s_i} \frac{b_{s_i}}{2^{k_{s_i}} p_{s_i}}
	\end{equation}
	for $n \in \nn$ and some subindices $s_1, \dots, s_n \in \nn \setminus \{j\}$ such that $s_1 < \dots < s_n$ and some coefficients $c_{s_1}, \dots, c_{s_n} \in \nn$. Since the $p_{s_1}$-adic valuation of the left-hand side of~\eqref{eq:auxi} is nonnegative, $p_{s_1} \mid c_{s_1}$. This implies that $c_{s_1} \ge p_{s_1}$. However, that means the right-hand side of the equality~\eqref{eq:auxi} is at least $\frac{b_{s_1}}{2^{k_{s_1}}} > 1$ while the left-hand side of the same equality is $\frac{b_j}{2^{k_j} p_j} < \frac{1}{2^{k_j}} \le 1$, a contradiction. Therefore $\mathcal{A}(M_\ell) \supseteq \big\{ \frac{r_n}{p_n} \mid n \in \nn \big\}$, and the claim is established.
	
	Proving that $M_\ell$ is not a Furstenberg monoid amounts to arguing that $1$ is not divisible by any atom. Assume, by way of contradiction, that $1$ is divisible by $\frac{r_j}{p_j}$ for some $j \in \nn$. Let $m$ be the largest positive integer such that $1 - m \frac{r_j}{p_j} \in M$. Therefore $1 - m \frac{r_j}{p_j}$ can be written as a sum of a dyadic rational and rational numbers of the form $\frac{r_k}{p_k}$ for positive integers $k \neq j$, which have nonnegative $p_j$-adic valuations. As a result, $1 - m \frac{r_j}{p_j}$ has nonnegative $p_j$-adic valuation. Since $v_{p_j}(r_j) = v_{p_j}(\frac{b_j}{2^{k_j}}) = 0$ by construction, we find that $p_j \mid m$. Thus, $1 - m \frac{r_j}{p_j} \le 1 - r_j < 0$, which is a contradiction. Hence $1$ is not divisible in $M_\ell$ by any atom, and so $M_\ell$ is not a Furstenberg monoid.
	
	Finally, we argue that $M_\ell$ is almost Furstenberg. To do this, fix $q \in M_\ell$. If $q$ is divisible by an atom $a$, then we are done: in this case, $q+0$ is divisible by $a$ in $M_\ell$ while $0$ is not. Otherwise, $q$ must be a dyadic rational. Let $c$ be any dyadic rational with $c > 1$, and take $i,j \in \nn$ such that $r_i = c$ and $r_j = q+c$. Then, $q+c$ is divisible by the atom $\frac{r_j}{p_j}$ in $M_\ell$. Assume, for the sake of contradiction, that $c$ is also divisible by the atom $\frac{r_j}{p_j}$ in $M_\ell$. Let $m$ be the largest positive integer such that $c - m \frac{r_j}{p_j} \in M_\ell$. Then, $c - m \frac{r_j}{p_j}$ can be written as a sum of a dyadic rational and rational numbers of the form $\frac{r_k}{p_k}$ for positive integers $k \neq j$, which have nonnegative $p_j$-adic valuations. Thus, $c - m \frac{r_j}{p_j}$ has nonnegative $p_j$-adic valuation. Since $v_{p_j}(r_j) = v_{p_j}(\frac{b_j}{2^{k_j}}) = 0$ by construction, we obtain that $p_j \mid m$. As a result, $c - m\frac{r_j}{p_j} \le c - r_j = -q$, a contradiction. Hence we conclude that $M_\ell$ is an almost Furstenberg Puiseux monoid that is not Furstenberg, which completes the proof.
\end{proof}

The monoid $M$ is called \emph{nearly Furstenberg} if there exists $c \in M$ such that for all $b \in M \setminus \uu(M)$ there exists $a \in \mathcal{A}(M)$ such that $a \mid_M b+c$ but $a \nmid_M c$. It follows directly from the definitions that every Furstenberg monoid is nearly Furstenberg. However, the converse of this statement does not hold: actually, there are Puiseux monoids that are both nearly Furstenberg and almost Furstenberg but not Furstenberg. The following theorem, which generalizes Proposition~\ref{prop:AF not F}, sheds some light upon this observation.

\begin{theorem} \label{thm:AF-NF not F} 
	There are infinitely many non-isomorphic  Puiseux monoids that are both almost Furstenberg and nearly Furstenberg but not Furstenberg.
\end{theorem}

\begin{proof}
	Let the sequences $(P_\ell)_{\ell \ge 1}$ and $(M_\ell)_{\ell \ge 1}$ be as in the proof of Proposition~\ref{prop:AF not F}. In light of the same proof, it suffices to argue that $M_\ell$ is nearly Furstenberg for every $\ell \in \nn$. To do so, fix $\ell \in \nn$, and then recall that $(r_n)_{n \ge 1}$ is a sequence of rationals and $(p_n)_{n \ge 1}$ is a sequence of primes in $P_\ell$ such that $\nn_0\big[\frac{1}{2}\big]_{> 1} = \{ r_n \mid n \in \nn \}$ and $p_n > \mathsf{n}(r_n)$ for every $n \in \nn$. Also, recall that
	\[
		M_\ell := \bigg\langle \nn_0\Big[\frac12 \Big] \bigcup \bigg\{ \frac{r_n}{p_n} \ \Big{|} \ n \in \nn \bigg\} \bigg\rangle.
	\]
	To argue that $M_\ell$ is nearly Furstenberg, fix a nonzero $q \in M_\ell$, and let us check that $1+q$ is divisible by an atom in $M_\ell$ (we have already seen in the proof of Proposition~\ref{prop:AF not F} that~$1$ is not divisible by any atom in $M_\ell$). If $q$ is divisible by an atom in $M_\ell$, then $1 + q$ is also divisible by an atom in $M_\ell$. Otherwise, $q$ must be a dyadic rational. However, any dyadic rational greater than $1$ is a multiple of some atom, so $1+q$ must be divisible by an atom in $M_\ell$. Thus, the Puiseux monoid~$M_\ell$ is nearly Furstenberg, which completes our proof.
\end{proof}

\begin{remark}
	In the context of ring theory, an example of an integral domain that is both nearly Furstenberg and almost Furstenberg but not Furstenberg was constructed in~\cite[Example~15]{nLL19}.
\end{remark}

It follows directly from the definitions that every almost Furstenberg monoid is quasi-Furstenberg and also that every nearly Furstenberg monoid is quasi-Furstenberg. We conclude this subsection revisiting Example~\ref{ex:a quasi-Furstenberg PM not atomic}, and showing that the quasi-Furstenberg Puiseux monoid discussed in the same is neither almost Furstenberg nor nearly Furstenberg.

\begin{example}
	Consider the Puiseux monoid $M := \big\langle \frac12, \frac1{3^n} \mid n \in \nn_0 \big\rangle$. We have already seen in Example~\ref{ex:a quasi-Furstenberg PM not atomic} that $M$ is a quasi-Furstenberg non-atomic Puiseux monoid with $\mathcal{A}(M) = \big\{ \frac{1}{2} \big\}$. We claim that $M$ is neither almost Furstenberg nor nearly Furstenberg. Observe that each nonzero atomic element of $M$ must be divisible by $\frac12$ in $M$ because $\frac12$ is the only atom of $M$. This, along with the fact that $\frac13$ is not a Furstenberg element in $M$, ensures that $\frac13$ is not almost Furstenberg in $M$. Therefore~$M$ is not an almost Furstenberg Puiseux monoid.

	Now assume, towards a contradiction, that there exists $c \in M$ such that $\frac12 \nmid_M c$ but $\frac12 \mid_M b+c$ for all nonzero $b \in M$. Set $N := \big\langle \frac1{3^n} \mid n \in \nn_0 \big\rangle$, which is a submonoid of $M$. Observe that $c \in N$. Also, because $\frac12 \nmid_M c$, we see that $c < 1$. Now we can choose $b \in N$ sufficiently small such that $b + c \in N$ and $b+c < 1$. Then we find that $b+c$ is not divisible by $\frac12$ in $M$: this is because $b+c - \frac12 < \frac12 = \min \big(\frac12 + M\big)$ and $b+c - \frac12 \notin N$. This gives the needed contradiction. Hence we conclude that the Puiseux monoid $M$ is a quasi-Furstenberg monoid that is neither an almost Furstenberg nor a nearly Furstenberg monoid.
\end{example}

\medskip
%%%%%%%%%%%%%%%%%%%%%%%%%%%%%%
\subsection{Almost Furstenberg VS Nearly Furstenberg}
In this subsection, we will show that there is no inclusion relation between the class of almost atomic Puiseux monoids and the class of nearly atomic Puiseux monoids. Let us start by showing that there are Puiseux monoids that is nearly Furstenberg but not almost Furstenberg.

\begin{theorem} \label{thm:NF not AF} %(cf. \cite[Example~15]{nLL19})
	There are infinitely many non-isomorphic  Puiseux monoids that are nearly Furstenberg but not almost Furstenberg.
\end{theorem}

\begin{proof}
	For each $p \in \pp$ with $p \ge 7$, consider the Puiseux monoid
	\[
		M_p := \bigg\langle \Big\{\frac1p \Big\} \bigcup \nn_0\Big[\frac12 \Big]^\bullet \bigcup \bigg( \frac12 - \frac1p + \nn_0\Big[\frac12 \Big]^\bullet \bigg) \bigg\rangle.
	\]
	We claim that $\mathcal{A}(M_p) = \big\{ \frac1p \big\}$. Since $ \nn_0\big[\frac12 \big] \subseteq M_p$, it follows that none of the elements in $ \nn_0\big[\frac12 \big]$ is an atom of $M_p$. In addition, for any $\frac12 - \frac1p + \frac{b}{2^c}$, where $b \in \nn$ and $c \in \nn_0$, we can write
	\[
		\frac12 - \frac1p + \frac{b}{2^c} = \Big( \frac12 - \frac1p + \frac{b}{2^{c+1}} \Big) + \frac{b}{2^{c+1}},
	\]
	from which we can infer that $\frac12 - \frac1p + \frac{b}{2^c}$ is not an atom of $M_p$. We proceed to argue that $\frac1p$ is an atom of $M_p$. Write $\frac1p = q+r$ for some $q,r \in M_p$ and assume, without loss of generality, that $p \mid \mathsf{d}(q)$. Since $\inf\big( \frac12 - \frac1p + \nn_0\big[\frac12 \big]^\bullet \big) = \frac12 - \frac1p > \frac1p$, we see that $q,r \notin  \frac12 - \frac1p + \nn_0\big[\frac12 \big]^\bullet$. This, along with the fact that $p \mid \mathsf{d}(q)$, ensures that $q \in \nn_0 \frac1p$. Therefore $q \in \big\{0,\frac1p \big\}$, and so $\frac1p$ is an atom. Thus, $\mathcal{A}(M_p) = \big\{ \frac1p \big\}$.
	
	Since $|\mathcal{A}(M_p)| = 1$, it follows from Proposition~\ref{prop:almost Furstenberg conditions} that $M_p$ is not almost Furstenberg. In order to argue that $M_p$ is nearly Furstenberg, we need to prove first the following claim.
	\smallskip
	
	\noindent \textit{Claim:} $\frac1p \nmid_{M_p} \frac12$.
	\smallskip
	
	\noindent \textit{Proof of Claim:} This is equivalent to showing that $\frac12 - \frac1p$ does not belong to $M_p$. Suppose, towards a contradiction, that this is not the case. Hence $\frac12 - \frac1p$ can be written as a nonnegative integer linear combination of defining generators. Since $\frac12 - \frac1p $ is strictly less than any element in the set $\frac12 - \frac1p + \nn_0\big[ \frac12 \big]^\bullet$, we see that $\frac12 - \frac1p \in \big\langle \big\{\frac1p \big\} \bigcup \nn_0\big[ \frac12 \big] \big\rangle$. Then we can write $\frac12 = c_0 \frac1p + \frac{b}{2^c}$ for some $c_0 \in \nn$ and $b,c \in \nn_0$. As the $p$-adic valuation of $\frac12$ is~$0$, it follows that $p \mid c_0$, which implies that the right-hand side of $\frac12 = c_0 \frac1p + \frac{b}{2^c}$ is at least $1$, a contradiction. Thus, the claim is established.
	\smallskip
	
	We are now in a position to prove that $M_p$ is nearly Furstenberg. It suffices to show that $\frac1p \mid_{M_p} \frac12 + m$ for each of the defining generators $m \in \big\{\frac1p \big\} \bigcup \nn_0\big[\frac12 \big]^\bullet \bigcup \big( \frac12 - \frac1p + \nn_0\big[\frac12 \big]^\bullet \big)$. It is clear that $\frac1p \mid_{M_p} \frac12 + \frac1p$. In addition, for each $m \in \nn_0\big[\frac12\big]^\bullet$, the fact that $\frac12 - \frac1p + m \in \frac12 - \frac1p +  \nn_0\big[\frac12\big]^\bullet \subseteq M_p$ ensures that $\frac1p \mid_{M_p} \frac12 + m$. Lastly, suppose that $m = \frac12 - \frac1p + \frac{b}{2^c}$ for some $b, c \in \nn$. In this case, we see that
	\[
		m + \frac12 - \frac1p = \Big( \frac12 - \frac1p + \frac{b}{2^{c+1}}\Big) +  \Big( \frac12 - \frac1p + \frac{b}{2^{c+1}}\Big) \in M_p,
	\]
	which implies that $\frac1p \mid_{M_p} \frac12 + m$. Hence $M_p$ is a nearly Furstenberg Puiseux monoid that is not almost Furstenberg.
	
	Finally, suppose, towards a contradiction, that $p$ and $q$ are two distinct primes in $\pp_{\ge 7}$ such that the Puiseux monoids $M_p$ and $M_q$ are isomorphic. Since isomorphism between Puiseux monoids are given by rational multiplication we can take $b,c \in \nn$ with $\gcd(b,c) = 1$ such that $c M_p = b M_q$. Since $\frac1p \in M_p$, it follows that $\frac cp \in b M_q$, which implies that $p \mid c$ because the $p$-adic valuation of any element of $M_q$ is nonnegative. Now since $1 \in M_q$, we can take $r \in M_p$ such that $c r = b$. Since $\gcd(b,c) = 1$, no prime distinct from $2$ and $p$ can divide $c$ (as it would also divide $b$), and so $c = 2^s p$ for some $s \in \nn_0$. Thus, $M_q = \frac cb M_p = 2^{s-t}\frac pq M_p$ and, after taking $N \in \nn$ large enough so that $2^{N+s-t} > p$, we obtain an element $r := \frac p{2^{N+s-t}q} \in M_q$ such that $r < \frac 1q$ and $q \mid \mathsf{d}(r)$, which is a contradiction (see the definition of $M_q$). Hence if $p$ and $q$ are distinct primes in $\pp_{\ge 7}$, then the Puiseux monoids $M_p$ and $M_q$ are not isomorphic.
\end{proof}

Let us turn now to show that there are almost Furstenberg Puiseux monoids that are not nearly Furstenberg.

\begin{theorem}
	There are infinitely many non-isomorphic Puiseux monoids that are almost Furstenberg but not nearly Furstenberg.
\end{theorem}

\begin{proof}
	For each $\ell \in \nn$, we let $P_\ell$ be an infinite set of odd primes such that the sets $P_k$ and $P_\ell$ are disjoint for all $k, \ell \in \nn$ with $k \neq \ell$. For each $\ell \in \nn$, we will construct a Puiseux monoid $M_\ell$ whose prime support is an infinite subset of $\{2\} \cup P_\ell$ such that $M_\ell$ is almost Furstenberg but not nearly Furstenberg. Since for any distinct $k, \ell \in \nn$ the symmetric difference of the prime supports of $M_k$ and $M_\ell$ is an infinite set, it will follow from part~(1) of Lemma~\ref{lem:prelim isomorphism} that the Puiseux monoids $M_k$ and $M_\ell$ will not be isomorphic.
	
	Fix $\ell \in \nn$, and let $(p_n)_{n \ge 1}$ be a sequence whose terms are pairwise distinct primes in $P_\ell$ such that $p_n \nmid 2^n-1$ for every $n \in \nn$.  Now consider the Puiseux monoid
	\[
		M_\ell := \bigg\langle \frac{1}{2^n},\frac{1}{p_n}\Big( 1 - \frac{1}{2^n} \Big) \ \Big{|} \ n \in \nn \bigg\rangle.
	\]
	For each $n \in \nn$, observe that the only defining generator of $M_\ell$ having negative $p_n$-adic valuation is $\frac{1}{p_n}(1-\frac{1}{2^n})$. Therefore $\frac{1}{p_n}(1-\frac{1}{2^n}) \in \mathcal{A}(M_\ell)$ for every $n \in \nn$. This, along with the fact that the rest of the defining generators (namely, the elements of $\big\{ \frac{1}{2^n} \mid n \in \nn\}$) are not atoms, ensures that
	\[
		\mathcal{A}(M_\ell) := \bigg\{ \frac{1}{p_n}\Big( 1 - \frac{1}{2^n} \Big) \ \Big{|} \ n \in \nn \bigg\}.
	\]
	Set $a_n := \frac{1}{p_n}(1-\frac{1}{2^n})$. We proceed to argue that $M_\ell$ is almost Furstenberg. To do so, we first prove that the non-Furstenberg elements of $M_\ell$ are the elements in $\zz\big[ \frac12 \big] \cap \big(0,\frac12 \big)$. First, note that $\nn_0\big[ \frac12 \big] \subseteq M_\ell$. Observe that if an element of $M_\ell$ is not Furstenberg, then it must be generated by the set $\big\{ \frac{1}{2^n} \mid n \in \nn \big\}$ and so it must be contained in $\nn_0\big[ \frac12 \big]$. We can see, on the other hand, that every element in $\nn_0\big[ \frac12 \big]_{\ge 1/2}$ is divisible by $\frac12$ in $M_\ell$, and so the fact that $\frac12$ is a Furstenberg element of $M_\ell$ (indeed, $\frac12 = p_1 a_1$) implies that every element in $\nn_0\big[ \frac12 \big]_{\ge 1/2}$ is a Furstenberg element. Hence every non-Furstenberg element of~$M_\ell$ belongs to $\zz\big[ \frac12 \big] \cap \big(0,\frac12 \big)$. Conversely, suppose that $q \in M_\ell$ is a Furstenberg element, and write
	\begin{equation} \label{eq:aux}
		q = \frac{c_0}{2^k} + \sum_{i=1}^N c_i a_i
	\end{equation}
	for some $N \in \nn$ and $k, c_0, \dots, c_N \in \nn_0$ such that $\sum_{i=1}^N c_i \ge 1$. Take $j \in \ldb 1,N \rdb$ such that $c_j \ge 1$. In this case, after applying $p_j$-adic valuation to both sides of the equality~\eqref{eq:aux}, we find that $p_j \mid c_j (2^j - 1)$ and so $p_j \mid c_j$. As a result, $q \ge c_j a_j\ge 1 - \frac{1}{2^j} \ge \frac12$. Hence no Furstenberg element of $M_\ell$ can be contained in $\zz\big[ \frac12 \big] \cap \big(0,\frac12 \big)$. Hence we conclude that set of non-Furstenberg elements of $M_\ell$ is $\zz\big[ \frac12 \big] \cap \big(0,\frac12 \big)$.
	
	Since each non-Furstenberg element is divisible by $\frac1{2^n}$ for some $n \in \nn$, arguing that $M_\ell$ is an almost Furstenberg element amounts to showing that $\frac1{2^n}$ is an almost Furstenberg element for every $n \in \nn$. Fix $n \in \nn$, and let us show that $\frac{1}{2^n}$ is an almost Furstenberg element. Observe that $c := 1-\frac{1}{2^n}$ is atomic because it is the sum of $p_n$ copies of the atom $a_n$. In addition, the atom $a_{n+1}$ divides the element $c + \frac{1}{2^n} = 1$ in $M_\ell$: indeed, $1 = p_{n+1} a_{n+1} + \frac{1}{2^{n+1}}$. Therefore we only need to verify that $a_{n+1} \nmid_{M_\ell} c$. It is enough to observe that if $c = \frac{c_0}{2^k} + \sum_{i=1}^N c_i a_i$ for some index $N \ge n+1$ and $k,c_0, \dots, c_N$, then after applying $p_{n+1}$-adic valuation to both sides of this equality we will find that $p_{n+1} \mid c_{n+1}$, and so the inequality $c < p_{n+1} a_{n+1}$ will ensure that $c_{n+1} = 0$. Hence $a_{n+1} \nmid_{M_\ell} c$. As a result, $M_\ell$ is almost Furstenberg.
	
	Finally, we will argue that $M_\ell$ is not nearly Furstenberg. Fix an arbitrary nonzero $c \in M_\ell$, and then write
	\begin{equation} \label{eq:temp}
		c = d_c + \sum_{i=1}^N c_i a_i
	\end{equation} 
	for some $d_c \in \nn_0\big[ \frac12\big]$ and $c_1, \dots, c_N \in \nn_0$. We can further assume that $c_k < p_k$ for every $k \in \ldb 1,N \rdb$ (otherwise, while $c_k \ge p_k$ for some $k \in \ldb 1, N \rdb$, we could just take away $p_k$ of these atoms, add their sum to $d_c$). Now take $i \in \nn$ such that none of the terms of the sequence $\big( 1 - \frac1{2^n}\big)_{n \ge 1}$ belongs to the interval $\big(d_c, d_c + \frac1{2^i} \big]$ (this is possible because the sequence $\big( 1 - \frac1{2^n}\big)_{n \ge 1}$ is strictly increasing). We will argue that $c + \frac{1}{2^i}$ is not divisible in $M_\ell$ by any atom that does not divide $c$. To do so, take an atom $a_j$ such that $a_j \nmid_{M_\ell} c$ and suppose, by way of contradiction, that $a_j \mid_{M_\ell} c + \frac1{2^i}$. Let $m$ be the largest positive integer such that $m a_j \mid_{M_\ell} c+ \frac1{2^i}$. Since $a_j \nmid_{M_\ell} c$, it follows that $v_{p_j}(c)\ge 0$, which implies that $v_{p_j}\big( c + \frac1{2^i}\big) \ge 0$. Therefore $p_j \mid m$, and so from the fact that $m a_j \mid_{M_\ell} c+ \frac1{2^i}$ we infer that $1 - \frac1{2^j} \mid_{M_\ell} c + \frac1{2^i}$. On the other hand, the fact that $a_j \nmid_{M_\ell} c$ ensures that $1 - \frac1{2^j} \nmid_{M_\ell} c$. Since $1 - \frac1{2^j} \nmid_{M_\ell} c$, the inequality $d_c < 1-\frac1{2^j}$. As none of the terms of the sequence $\big( 1 - \frac1{2^n}\big)_{n \ge 1}$ belongs to the interval $\big(d_c, d_c + \frac1{2^i} \big]$, it follows that $d_c + \frac{1}{2^i} < 1 - \frac1{2^j}$. Since $1 - \frac1{2^j} \mid_{M_\ell} c + \frac1{2^i}$, we can write $ c + \frac1{2^i}$ as in the second equality of the expression
	\begin{equation} \label{eq:temp1}
		\big(d_c + \frac1{2^i} \big) + \sum_{i=1}^N c_i a_i = c + \frac1{2^i} = d' + \sum_{i=1}^{N'} c'_i a_i,
	\end{equation}
	where $d' \in \nn_0\big[ \frac12 \big]$ and $c'_1, \dots, c'_{N'} \in \nn_0$ satisfy that $d' \ge 1 - \frac1{2^j}$ and $c'_k < p_k$ for every $k \in \ldb 1, N' \rdb$ (after inserting zero coefficients if necessary, we can assume that $N' \ge N$). After applying $p_i$-adic valuation to~\eqref{eq:temp1} for every $i \in \ldb 1, N' \rdb$, we find that $c'_i = c_i$ for every $i \in \ldb 1, N \rdb$ and also that $c'_i = 0$ for every $i \in \ldb N+1, N' \rdb$. Hence $d' = d_c + \frac1{2^i} < 1 - \frac1{2^j}$, which is a contradiction. Hence $a_j \nmid_{M_\ell} c + \frac1{2^i}$. Since every atom that does not divide $c$ in $M_\ell$ cannot divide $c + \frac1{2^i}$ in $M_\ell$, the fact that $c$ was arbitrarily chosen guarantees that $M_\ell$ is not nearly Furstenberg. Thus, we conclude that the Puiseux monoid $M_\ell$ is almost Furstenberg but not nearly Furstenberg.
\end{proof}

\bigskip
%%%%%%%%%%%%%%%%%%%%%%%%%%%%%%%
%%%%%%%%%%%%%%%%%%%%%%%%%%%%%%%
\section{Two Final Examples in Connection with Atomicity}
\label{sec:atomicity}

Let $M$ be a monoid. We say that an element $b \in M$ is \emph{almost atomic} if there exists an atomic element $c \in M$ such that $b+c$ is atomic. Observe that every invertible element is almost atomic. The monoid $M$ is called \emph{almost atomic} if every element of $M$ is almost atomic. Also, the monoid $M$ is called \emph{nearly atomic} if there exists $c \in M$ such that for each $b \in M \setminus \uu(M)$, the element $b+c$ is atomic. We can mimic the proof of \cite[Lemma~7]{nLL19} to obtain that every nearly atomic monoid is almost atomic. The notion of almost atomicity was introduced by Boynton and Coykendall~\cite{BC15}, while the notion of near atomicity was introduced by Lebowitz-Lockard~\cite{nLL19}.

The property of being Furstenberg and that of being nearly atomic are both weaker than the property of being atomic. It is natural to wonder whether one of the two former properties implies the other one in the class of Puiseux monoids. As the next two examples show, none of these properties implies the other one. In the following example we exhibit a Furstenberg monoid that is not even almost atomic (in particular, it is not nearly atomic).

\begin{example} \label{Furstenberg PM not AA}
	First we construct a Puiseux monoid that is Furstenberg but not nearly atomic (we have already seen in Example~\ref{ex:Furstenberg not QA} a Furstenberg monoid that is not even quasi-atomic; however, the corresponding monoid is not a Puiseux monoid). Consider the Puiseux monoid
	\[
		M := \Big\langle \frac1p \ \Big{|} \ p \in \pp_{\ge 3} \Big\rangle \bigcup \qq_{\ge 1}. 
	\]
	Observe that none of the elements in $\qq_{\ge 1}$ can be an atom of $M$ because they are each divisible by $\frac1p$ in $M$ for sufficiently large odd primes~$p$. On the other hand, $\frac1p \in \mathcal{A}(M)$ for every $p \in \pp_{\ge 3}$ because $\frac1p$ is the smallest element whose $p$-adic valuation is negative and, therefore, it cannot be written as the sum of smaller elements. Thus,
	\[
		\mathcal{A}(M) = \bigg\{ \frac1p \ \Big{|} \ p \in \pp_{\ge 3} \bigg\}.
	\]
	Therefore each nonzero element of $M$ is divisible by an atom in $M$, which implies that $M$ is a Furstenberg monoid. On the other hand, observe that $M$ is almost atomic if and only if every element of $M$ can be written as the difference of two atomic elements. In addition, note that differences of atomic elements will have nonnegative $2$-adic valuations as atoms have have nonnegative $2$-adic valuations. As a result, we infer that $\frac32$ is not a difference of atomic elements and, therefore, $M$ is not almost atomic. We conclude that the Puiseux monoid $M$ is Furstenberg but not almost atomic (and so not nearly atomic).
\end{example}

We proceed to construct a nearly atomic Puiseux monoid that is not Furstenberg.

\begin{example}
	For each $x \in \nn$, we let $\ell_2(x)$ denote the largest power of $2$ less than $x$. Let $(o_n)_{n \ge 1}$ denote the strictly increasing sequence whose terms are the odd positive integers greater than $1$, and let $(p_n)_{n \ge 1}$ denote the strictly increasing sequence whose terms are the primes greater than $3$. Notice that $o_i < p_i$ for every $i \in \nn$ as prime numbers greater than $3$ are a subset of the odd numbers. Now consider the Puiseux monoid
	\[
		M := \bigg\langle \frac13, \frac{1}{2^n}, \frac{o_n}{\ell_2(o_n)p_n} \ \Big{|} \ n \in \nn \bigg\rangle.
	\]
	Let us argue first that $\mathcal{A}(M) = \big\{ \frac13, \frac{o_n}{\ell_2(o_n)p_n} \mid n \in \nn \big\}$. For all $i \in \nn$, the inequality $o_i < p_i$ ensures that the defining generator $\frac{o_n}{\ell_2(o_n)p_n}$ has a prime factor $p$ in the denominator that is not shared by any other defining generator: therefore $\frac{o_n}{\ell_2(o_n)p_n}$ is the only defining generator with negative $p$-adic valuation, and so it must belong to $\mathcal{A}(M)$. The same argument can be invoked to show that $\frac 13 \in \mathcal{A}(M)$. On the other hand, for each $n \in \nn_0$, the fact that $\frac1{2^n} = \frac1{2^{n+1}} + \frac1{2^{n+1}}$ guarantees that $\frac{1}{2^n} \notin \mathcal{A}(M)$. As a result,
	\[
		\mathcal{A}(M) = \bigg\{ \frac13, \frac{o_n}{\ell_2(o_n)p_n} \ \Big{|} \ n \in \nn \bigg\}.
	\]
		
	Let us prove now that $M$ is nearly atomic. To do so, first observe that $1$ is an atomic element: indeed, $1$ is the sum of three copies of the atom $\frac13$. Thus, proving that $M$ is nearly atomic amounts to showing that $1 + \frac{x}{2^y}$ is atomic for all $x \in \nn$ and $y \in \nn_0$ (as adding any of the other defining generators will preserve the atomic condition). Fix $x \in \nn$ and $y \in \nn_0$, and note that $1 + \frac{x}{2^y}$ can be written as $\frac{a}{2^b}$ for some $a,b \in \nn_0$ such that $a>2^b$ and $a \in 2 \nn_0 + 1$. Suppose that $a$ is the $n$-th term of the sequence $(o_n)_{n \ge 1}$, and consider the defining generator $\frac{a}{\ell_2(a)p_n}$, which is an atom of $M$. It follows from the definition of $\ell_2(a)$ that $2^b \mid \ell_2(a)$ and, therefore, $\frac{a}{2^b}$ is a multiple of the atom $\frac{a}{\ell_2(a)p_n}$. Thus,~$M$ is nearly atomic.

	Finally, we verify that $M$ is not a Furstenberg monoid. It suffices to show that $\frac12$ is not divisible by any atom in $M$. Suppose, for the sake of a contradiction, that $a_0 \mid_M \frac12$ for some $a_0 \in \mathcal{A}(M)$. Let $p_0 \in \pp$ be the odd prime dividing $\mathsf{d}(a_0)$. Since $a_0 \mid_M \frac12$, we can write $\frac12 = \frac{b}{2^c} + \sum_{i=0}^N c_i a_i$ for distinct atoms $a_0, \dots, a_N$ and coefficients $c_0, \dots, c_N \in \nn_0$ such that $c_0 \neq 0$. Since the atoms $a_0, \dots, a_N$ are pairwise distinct, after applying $p_0$-adic valuation to both sides of the equality $\frac12 = \frac{b}{2^c} + \sum_{i=0}^N c_i a_i$, we find that $p_0 \mid c_0$. Hence $c_0 a_0 \ge p_0 a_0 \ge 1 > \frac12$, which is a contradiction. Hence we conclude that $M$ is a nearly atomic (and so a nearly Furstenberg) Puiseux monoid that is not Furstenberg.
\end{example}

\bigskip
%%%%%%%%%%%%%%%
%%%%%%%%%%%%%%%
\section*{Acknowledgments}

We would like to thank our mentor Dr. Felix Gotti for guiding us in both our research and the writing process of this paper. We would also like to thank the MIT PRIMES program for the research opportunity.

\bigskip
%%%%%%%%%%%%%%
%%%%%%%%%%%%%%

\end{document}